\documentclass[oneside,english]{amsart}
\usepackage[T1]{fontenc}
\usepackage[latin9]{inputenc}
\usepackage{amsthm}
\usepackage{amstext}
\usepackage{amssymb}
\usepackage{graphicx}

\makeatletter
\numberwithin{equation}{section}
\numberwithin{figure}{section}
\theoremstyle{plain}
\newtheorem{thm}{\protect\theoremname}
  \theoremstyle{definition}
  \newtheorem{defn}[thm]{\protect\definitionname}
  \theoremstyle{plain}
  \newtheorem{lem}[thm]{\protect\lemmaname}
  \theoremstyle{definition}
  \newtheorem{example}[thm]{\protect\examplename}
  \theoremstyle{plain}
  \newtheorem{cor}[thm]{\protect\corollaryname}
  \theoremstyle{remark}
  \newtheorem*{claim*}{\protect\claimname}

\makeatother

\usepackage{babel}
  \providecommand{\claimname}{Claim}
  \providecommand{\corollaryname}{Corollary}
  \providecommand{\definitionname}{Definition}
  \providecommand{\examplename}{Example}
  \providecommand{\lemmaname}{Lemma}
\providecommand{\theoremname}{Theorem}

\begin{document}

\title{Epic substructures and primitive positive functions}

\author{Miguel Campercholi}
\begin{abstract}
For $\mathbf{A}\leq\mathbf{B}$ first order structures in a class
$\mathcal{K}$, say that $\mathbf{A}$ is an epic substructure of
$\mathbf{B}$ in $\mathcal{K}$ if for every $\mathbf{C}\in\mathcal{K}$
and all homomorphisms $g,g^{\prime}:\mathbf{B}\rightarrow\mathbf{C}$,
if $g$ and $g'$ agree on $A$, then $g=g'$. We prove that $\mathbf{A}$
is an epic substructure of $\mathbf{B}$ in a class $\mathcal{K}$
closed under ultraproducts if and only if $A$ generates $\mathbf{B}$
via operations definable in $\mathcal{K}$ with primitive positive
formulas. Applying this result we show that a quasivariety of algebras
$\mathcal{Q}$ with an $n$-ary near-unanimity term has surjective
epimorphisms if and only if $\mathbb{SP}_{n}\mathbb{P}_{u}(\mathcal{Q}_{RSI})$
has surjective epimorphisms. It follows that if $\mathcal{F}$ is
a finite set of finite algebras with a common near-unanimity term,
then it is decidable whether the (quasi)variety generated by $\mathcal{F}$
has surjective epimorphisms.
\end{abstract}


\maketitle

\section{Introduction}

Let $\mathcal{K}$ be a class of first order structures in the same
signature, and let $\mathbf{A},\mathbf{B}\in\mathcal{K}$. We say
that $\mathbf{A}$ is an \emph{epic substructure} of $\mathbf{B}$
in $\mathcal{K}$ provided that $\mathbf{A}$ is a substructure of
$\mathbf{B}$, and for every $\mathbf{C}\in\mathcal{K}$ and all homomorphisms
$g,g^{\prime}:\mathbf{B}\rightarrow\mathbf{C}$ such that $g|_{A}=g^{\prime}|_{A}$,
we have $g=g^{\prime}$. That is, if $g$ and $g'$ agree on $A$,
then they must agree on all of $B$. At first glance the definition
may suggest that $A$ generates $\mathbf{B}$, but on closer inspection
this does not make sense. As $\mathbf{A}$ is a substructure of $\mathbf{B}$,
generating with $A$ will yield exactly $\mathbf{A}$. However, as
the main result of this article shows, the intuition that $A$ acts
as a set of generators of $\mathbf{B}$ is not far off. In fact, if
$\mathcal{K}$ is closed under ultraproducts, we prove that $A$ actually
``generates'' $\mathbf{B}$, only that the generation is not through
the fundamental operations but rather through primitive positive definable
functions. Let's take a look at an example. Write $\mathcal{D}$ for
the class of bounded distributive lattices. There are several ways
to show that both of the three-element chains contained in the bounded
distributive lattice $\mathbf{B}:=\mathbf{2}\times\mathbf{2}$ are
epic substructures of $\mathbf{B}$ in $\mathcal{D}$. One way to
do this is via definable functions. Note that the formula 
\[
\varphi(x,y):=x\wedge y=0\,\&\, x\vee y=1
\]
defines the complement (partial) operation in every member of $\mathcal{D}$.
Let $\mathbf{A}$ be the sublattice of $\mathbf{B}$ with universe
$\{\left\langle 0,0\right\rangle ,\left\langle 0,1\right\rangle ,\left\langle 1,1\right\rangle \}$,
and suppose there are $\mathbf{C}\in\mathcal{D}$ and $g,g^{\prime}:\mathbf{B}\rightarrow\mathbf{C}$
such that $g|_{A}=g^{\prime}|_{A}$. Clearly $\mathbf{B}\vDash\varphi(\left\langle 0,1\right\rangle ,\left\langle 1,0\right\rangle )$,
and since $\varphi$ is open and positive, it follows that $\mathbf{C}\vDash\varphi(g\langle0,1\rangle,g\langle1,0\rangle)$
and $\mathbf{C}\vDash\varphi(g'\langle0,1\rangle,g'\langle1,0\rangle)$.
Now $\varphi(x,y)$ defines a function in $\mathbf{C}$, and $g\langle0,1\rangle=g'\langle0,1\rangle$,
so $g\langle1,0\rangle=g'\langle1,0\rangle$. Theorem \ref{epic sii pp definable}
below says that every epic substructure in a class closed under ultraproducts
is of this nature (although the formulas defining the generating operations
may be primitive positive).

The notion of epic substructure is closely connected to that of epimorphism.
Recall that a homomorphism $h:\mathbf{A}\rightarrow\mathbf{B}$ is
a $\mathcal{K}$\emph{-epimorphism} if for every $\mathbf{C}\in\mathcal{K}$
and homomorphisms $g,g^{\prime}:\mathbf{B}\rightarrow\mathbf{C}$,
if $gh=g^{\prime}h$ then $g=g^{\prime}$. That is, $h$ is right-cancellable
in compositions with $\mathcal{K}$-morphisms. Of course every surjective
homomorphism is an epimorphism, but the converse is not true. Revisiting
the example above, the inclusion of the three-element chain $\mathbf{A}$
into $\mathbf{2}\times\mathbf{2}$ is a $\mathcal{D}$-epimorphism.
This also illustrates the connection between epic substructures and
epimorphisms. It is easily checked that $\mathbf{A}$ is an epic substructure
of \textbf{$\mathbf{B}$} in $\mathcal{K}$ if and only if the inclusion
$\iota:A\rightarrow B$ is a $\mathcal{K}$-epimorphism. A class $\mathcal{K}$
is said to have the \emph{surjective epimorphisms }if every $\mathcal{K}$-epimorphism
is surjective. Although this property is of an algebraic (or categorical)
nature it has an interesting connection with logic. When $\mathcal{K}$
is the algebraic counterpart of an algebraizable logic $\vdash$ then:
$\mathcal{K}$ has surjective epimorphisms if and only if $\vdash$
has the (infinite) Beth property (\cite[Thm. 3.17]{beth-block-hoogland}).
For a thorough account on the Beth property in algebraic logic see
\cite{beth-block-hoogland}. We don't go into further details on this
topic as the focus of the present article is on the algebraic and
model theoretic side.

The paper is organized as follows. In the next section we establish
our notation and the preliminary results used throughout. Section
\ref{sec:Main-Theorem} contains our characterization of epic substructures
(Theorem \ref{epic sii pp definable}), the main result of this article.
We also take a look here at the case where $\mathcal{K}$ is a finite
set of finite structures. In Section \ref{sec:Checking-for-epic}
we show that checking for the presence of proper epic subalgebras
(or, equivalently, surjective epimorphisms) in certain quasivarieties
can be reduced to checking in a subclass of the quasivariety. An interesting
application of these results is that if $\mathcal{F}$ is a finite
set of finite algebras with a common near-unanimity term, then it
is decidable whether the quasivariety generated by $\mathcal{F}$
has surjective epimorphisms (see Corollary \ref{cor:NU implica decidible}).

\section{Preliminaries and Notation}

Let $\mathcal{L}$ be a first order language and $\mathcal{K}$ a
class of $\mathcal{L}$-structures. We write $\mathbb{I},\mathbb{S},\mathbb{H},\mathbb{P}$
and $\mathbb{P}_{u}$ to denote the class operators for isomorphisms,
substructures, homomorphic images, products and ultraproducts, respectively.
We write $\mathbb{V}(\mathcal{K})$ for the variety generated by $\mathcal{K}$,
that is $\mathbb{HSP}(\mathcal{K})$; and with $\mathbb{Q}(\mathcal{K})$
we denote the quasivariety generated by $\mathcal{K}$, i.e., $\mathbb{ISPP}_{u}(\mathcal{K})$. 
\begin{defn}
Let $\mathbf{A},\mathbf{B}\in\mathcal{K}$.
\begin{itemize}
\item $\mathbf{A}$ is an \emph{epic substructure} \emph{of $\mathbf{B}$
in} $\mathcal{K}$ if $\mathbf{A}\leq\mathbf{B}$, and for every $\mathbf{C}\in\mathcal{K}$
and all homomorphisms $g,g^{\prime}:\mathbf{B}\rightarrow\mathbf{C}$
such that $g|_{A}=g^{\prime}|_{A}$, we have $g=g^{\prime}$. Notation:
$\mathbf{A}\leq_{e}\mathbf{B}$ in $\mathcal{K}$.
\item A homomorphism $h:\mathbf{A}\rightarrow\mathbf{B}$ is a $\mathcal{K}$\emph{-epimorphism}
if for every $\mathbf{C}\in\mathcal{K}$ and homomorphisms $g,g^{\prime}:\mathbf{B}\rightarrow\mathbf{C}$,
if $gh=g^{\prime}h$ then $g=g^{\prime}$.
\end{itemize}
\end{defn}
We say that $\mathbf{A}$ is a \emph{proper} epic substructure of
$\mathbf{B}$ in $\mathcal{K}$ (and write $\mathbf{A}<_{e}\mathbf{B}$
in $\mathcal{K}$), if $\mathbf{A}\leq_{e}\mathbf{B}$ in $\mathcal{K}$
and $\mathbf{A}\neq\mathbf{B}$.

The next lemma explains the connection between epic substructures
and epimorphisms.
\begin{lem}
If $h:\mathbf{A}\rightarrow\mathbf{B}$ with $\mathbf{A},\mathbf{B},h(\mathbf{A})\in\mathcal{K}$,
then t.f.a.e.:
\begin{enumerate}
\item $h$ is a $\mathcal{K}$-epimorphism.
\item The inclusion map $\iota:h(\mathbf{A})\rightarrow\mathbf{B}$ is a
$\mathcal{K}$-epimorphism.
\item $h(\mathbf{A})\leq_{e}\mathbf{B}$ in $\mathcal{K}$.
\end{enumerate}
\end{lem}
\begin{proof}
Immediate from the definitions.
\end{proof}
Here are some straightforward facts used in the sequel.
\begin{lem}
Let $\mathbf{A},\mathbf{B}\in\mathcal{K}$.
\begin{enumerate}
\item $\mathbf{A}\leq_{e}\mathbf{B}$ in $\mathcal{K}$ iff $\mathbf{A}\leq_{e}\mathbf{B}$
in $\mathbb{ISP}(\mathcal{K})$
\item Let $\mathbf{A}\leq_{e}\mathbf{B}$ in $\mathcal{K}$ and suppose
$h:\mathbf{B}\rightarrow\mathbf{C}$ is such that $h(\mathbf{A}),h(\mathbf{B})\in\mathcal{K}$.
Then $h(\mathbf{A})\leq_{e}h(\mathbf{B})$ in $\mathcal{K}$.
\item Let $\mathcal{Q}$ be a quasivariety. T.f.a.e.:

\begin{enumerate}
\item $\mathcal{Q}$ has surjective epimorphisms.
\item For all $\mathbf{A},\mathbf{B}\in\mathcal{Q}$ we have that $\mathbf{A}\leq_{e}\mathbf{B}$
in $\mathcal{Q}$ implies $\mathbf{A}=\mathbf{B}$.
\end{enumerate}
\end{enumerate}
\end{lem}

\section{Main Theorem\label{sec:Main-Theorem}}

Recall that a \emph{primitive positive }(p.p.\ for brevity) formula
is one of the form $\exists\bar{y}\,\alpha(\bar{x},\bar{y})$ with
$\alpha(\bar{x},\bar{y})$ a finite conjunction of atomic formulas.
We shall need the following fact. 
\begin{lem}
\label{toda pp implica homo}(\cite[Thm. 6.5.7]{Hodges1993}) Let
$\mathbf{A},\mathbf{B}$ be $\mathcal{L}$-structures. T.f.a.e.:
\begin{enumerate}
\item Every primitive positive $\mathcal{L}$-sentence that holds in $\mathbf{A}$
holds in $\mathbf{B}$.
\item There is a homomorphism from $\mathbf{A}$ into an ultrapower of $\mathbf{B}$.
\end{enumerate}
\end{lem}
Let $\mathcal{K}$ be a class of $\mathcal{L}$-structures. We say
that the $\mathcal{L}$-formula $\varphi(x_{1},\dots x_{n},y_{1},\dots,y_{m})$
\emph{defines a function} in $\mathcal{K}$ if 
\[
\mathcal{K}\vDash\forall\bar{x},\bar{y},\bar{z}\ \varphi\left(\bar{x},\bar{y}\right)\wedge\varphi\left(\bar{x},\bar{z}\right)\rightarrow\bigwedge_{j=1}^{m}y_{j}=z_{j}.
\]
In that case, for each $\mathbf{A}\in\mathcal{K}$ we write $[\varphi]^{\mathbf{A}}$
to denote the $n$-ary partial function defined by $\varphi$ in $\mathbf{A}$.

If $X$ is a set disjoint with $\mathcal{L}$, we write $\mathcal{L}_{X}$
to denote the language obtained by adding the elements in $X$ as
new constant symbols to $\mathcal{L}$. If $\mathbf{B}$ is an $\mathcal{L}$-structure
and $A$ is a subset of $B$, let $\mathbf{B}_{A}$ be the expansion
of $\mathbf{B}$ to $\mathcal{L}_{A}$ where each new constant names
itself. If $\mathcal{L}\subseteq\mathcal{L}^{+}$ and $\mathbf{A}$
is an $\mathcal{L}^{+}$-model, let $\mathbf{A}|_{\mathcal{L}}$ denote
the reduct of $\mathbf{A}$ to $\mathcal{L}$.

Next we present the main result of this article.
\begin{thm}
\label{epic sii pp definable}Let $\mathcal{K}$ be a class closed
under ultraproducts and $\mathbf{A}\leq\mathbf{B}$ structures. T.f.a.e.:
\begin{enumerate}
\item $A$ is an epic subalgebra of $\mathbf{B}$ in $\mathcal{K}$.
\item For every $b\in B$ there are a primitive positive formula $\varphi\left(\bar{x},y\right)$
and $\bar{a}$ from $A$ such that:

\begin{enumerate}
\item $\varphi\left(\bar{x},y\right)$ defines a function in $\mathcal{K}$
\item $[\varphi]^{\mathbf{B}}(\bar{a})=b$. 
\end{enumerate}
\end{enumerate}
\end{thm}
\begin{proof}
(1)$\Rightarrow$(2). We can assume that $\mathcal{K}$ is axiomatizable
(replacing $\mathcal{K}$ by $\mathbb{I}\mathbb{S}(\mathcal{K})$
if necessary). Suppose $\mathbf{A}\leq_{e}\mathbf{B}$ in $\mathcal{K}$
and let $b\in B$. Define
\[
\Sigma\left(x\right):=\{\varphi\left(x\right)\mid\varphi\left(x\right)\text{ is a p.p.\ formula of }\mathcal{L}_{A}\text{ and }\mathbf{B}_{A}\vDash\varphi\left(b\right)\}\text{,}
\]
Let $c,d$ be two new constant symbols and take
\[
\mathcal{K}^{*}:=\{\mathbf{M}\mid\mathbf{M}\text{ is a }\mathcal{L}_{A}\cup\{c,d\}\text{-model and }\mathbf{M}|_{\mathcal{L}}\in\mathcal{K}\}\text{.}
\]
Let $\mathbf{C}$ be a model of $\mathcal{K}^{*}$ such that $\mathbf{C}\vDash\Sigma(c)\cup\Sigma(d)$.
By Lemma \ref{toda pp implica homo}, there are elementary extensions
$\mathbf{E},\mathbf{E}^{\prime}$ of $\mathbf{C}$. and homomorphisms
\begin{align*}
h & :\mathbf{B}_{A}\rightarrow\mathbf{E}|_{\mathcal{L}_{A}}\\
h^{\prime} & :\mathbf{B}_{A}\rightarrow\mathbf{E}^{\prime}|_{\mathcal{L}_{A}}
\end{align*}
such that $h(b)=c^{\mathbf{C}}$ and $h^{\prime}(b)=d^{\mathbf{C}}$.
The elementary amalgamation theorem \cite[Thm. 6.4.1]{Hodges1993}
provides us with an algebra $\mathbf{D}$ and elementary embeddings
$g:\mathbf{E}\rightarrow\mathbf{D}$, $g^{\prime}:\mathbf{E}^{\prime}\rightarrow\mathbf{D}$
such that $g$ and $g^{\prime}$ agree on $C$. Next, observe that
\begin{align*}
gh & :\mathbf{B}\rightarrow\mathbf{D}|_{\mathcal{L}}\\
g^{\prime}h^{\prime} & :\mathbf{B}\rightarrow\mathbf{D}|_{\mathcal{L}}
\end{align*}
are homomorphisms that agree on $A$, and since $\mathbf{D}|_{\mathcal{L}}\in\mathcal{K}$
we must have
\[
gh=g^{\prime}h^{\prime}\text{.}
\]
In particular $gh(b)=g^{\prime}h^{\prime}(b)$, which is $g(c^{\mathbf{C}})=g^{\prime}(d^{\mathbf{C}})$.
So, as $g$ is 1-1, and $g$ and $g^{\prime}$ are the same on $C$
we have $c^{\mathbf{C}}=d^{\mathbf{C}}$.

Thus we have shown
\[
\mathcal{K}^{*}\vDash{\textstyle \bigwedge}\left(\Sigma\left(c\right)\cup\Sigma\left(d\right)\right)\rightarrow c=d\text{.}
\]
By compactness (and using that the conjunction of p.p.\ formulas
is equivalent to a p.p.\ formula), there is single p.p.\ $\mathcal{L}$-formula
$\varphi\left(\bar{x},y\right)$ such that
\[
\mathcal{K}^{*}\vDash\varphi(\bar{a},c)\wedge\varphi(\bar{a},d)\rightarrow c=d\text{,}
\]
and hence
\[
\mathcal{K}\vDash\forall\bar{x},y,z\ \varphi(\bar{x},z)\wedge\varphi(\bar{x},z)\rightarrow y=z\text{.}
\]
This completes the proof of (1)$\Rightarrow$(2).

(2)$\Rightarrow$(1). Suppose (2) holds for $\mathbf{A}$, $\mathbf{B}$
and $\mathcal{K}$. Let $\mathbf{C}\in\mathcal{K}$ and $h,h^{\prime}:\mathbf{B}\rightarrow\mathbf{C}$
homomorphisms agreeing on $A$. Fix $b\in B$. There are a p.p.\ formula
$\varphi\left(\bar{x},y\right)$ and $\bar{a}$ elements from $A$
such that
\begin{align*}
\mathbf{B} & \vDash\varphi(\bar{a},b)\\
\mathcal{K} & \vDash\forall\bar{x},y,z\ \varphi(\bar{x},y)\wedge\varphi(\bar{x},z)\rightarrow y=z\text{.}
\end{align*}
Hence
\[
\mathbf{C}\vDash\varphi(h\bar{a},hb)\wedge\varphi(h^{\prime}\bar{a},h^{\prime}b)\text{,}
\]
and as $h\bar{a}=h^{\prime}\bar{a}$ we have $hb=h^{\prime}b$.
\end{proof}
It is worth noting that (2)$\Rightarrow$(1) in Theorem \ref{epic sii pp definable}
always holds, i.e., it does not require for $\mathcal{K}$ to be closed
under ultraproducts. On the other hand, as the upcoming example shows,
the implication (1)$\Rightarrow$(2) may fail if $\mathcal{K}$ is
not closed under ultraproducts.
\begin{example}
Let $\mathcal{L}=\{s,0\}$ where $s$ is a binary function symbol
and $0$ a constant. Let $\mathbf{B}$ be the $\mathcal{L}$-structure
with universe $\omega\cup\{\omega\}$ such that $0^{\mathbf{B}}=0$
and 
\[
s^{\mathbf{B}}(a,b)=\begin{cases}
0 & \mbox{if }b=a+1,\\
1 & \mbox{otherwise.}
\end{cases}
\]
Take $\mathbf{A}$ the subalgebra of $\mathbf{B}$ with universe $\omega$.
It is easy to see that the identity is the only endomorphism of $\mathbf{B}$.
Thus, in particular, we have that $\mathbf{A}\leq_{e}\mathbf{B}$
in $\{\mathbf{B}\}$. We prove next that there is no p.p.\ formula
with parameters from $A$ defining $\omega$ in $\mathbf{B}$. Take
$\mathcal{L}^{+}:=\mathcal{L}_{B}\cup\{\omega'\}$, where $\omega'$
is a new constant, and let $\Gamma$ be the $\mathcal{L}^{+}$-theory
obtained by adding to the elementary diagram of $\mathbf{B}$ the
following sentences: 
\[
\{s(n,\omega')=1\mid n\in\omega\}\cup\{s(\omega',n)=1\mid n\in\omega\}\cup\{\omega\neq\omega'\}.
\]
It is a routine task to show that $\Gamma$ is consistent. Fix a model
$\mathbf{C}$ of $\Gamma$ and define $h,h':\mathbf{B}\rightarrow\mathbf{C}$
by $h(n)=h'(n)=n^{\mathbf{C}}$ for all $n\in\omega$, $h(\omega)=\omega^{\mathbf{C}}$
and $h'(\omega)=\omega'^{\mathbf{C}}$. Again, it is easy to see that
$h$ and $h'$ are homomorphisms from $\mathbf{B}$ to $\mathbf{C}|_{\mathcal{L}}$.
Since they agree on $A$ and $h(\omega)\neq h'(\omega)$, we conclude
that there is no p.p. formula with parameters from $A$ defining $\omega$
in $\mathbf{B}$.
\end{example}

\subsection{The finite case}

When $\mathcal{K}$ is (up to isomorphisms) a finite set of finite
structures, we can sharpen Theorem \ref{epic sii pp definable}. In
this case it is possible to avoid the existential quantifiers in the
definable functions at the cost of adding parameters from $\mathbf{B}$.
\begin{thm}
\label{thm:epicas en clases finitas}Let $\mathcal{K}$ be (up to
isomorphisms) a finite set of finite structures, and let $\mathbf{A}\leq\mathbf{B}$
be finite. T.f.a.e.:
\begin{enumerate}
\item $\mathbf{A}$ is an epic substructure of $\mathbf{B}$ in $\mathcal{K}$.
\item For every $b_{1}\in B$ there are a finite conjunction of atomic formulas
$\alpha(\bar{x},\bar{y})$, $a_{1},\dots,a_{n}\in A$ and $b_{2},\dots,b_{m}\in B$,
with $m\geq1$, such that

\begin{enumerate}
\item $\alpha(\bar{x},\bar{y})$ defines a function in $\mathcal{K}$
\item $[\alpha]^{\mathbf{B}}(\bar{a})=\bar{b}$.
\end{enumerate}
\item For every $b\in B$ there are a primitive positive formula $\varphi\left(\bar{x},y\right)$
and $\bar{a}$ from $A$ such that:

\begin{enumerate}
\item $\varphi\left(\bar{x},y\right)$ defines a function in $\mathcal{K}$
\item $[\varphi]^{\mathbf{B}}(\bar{a})=b$. 
\end{enumerate}
\end{enumerate}
\end{thm}
\begin{proof}
(1)$\Rightarrow$(2). If $b_{1}\in A$ the formula $x_{1}=y_{1}$
does the job. Suppose $b_{1}\notin A$, and let $a_{1},\dots,a_{n}$
and $b_{1},\dots,b_{m}$ be enumerations of $A$ and $B\setminus A$
respectively. Let 
\[
\Delta(\bar{x},\bar{y}):=\{\delta(\bar{x},\bar{y})\mid\delta(\bar{x},\bar{y})\mbox{ is an atomic formula and }\mathbf{B}\vDash\delta(\bar{a},\bar{b})\}.
\]
Since $\mathcal{K}$ is a finite set of finite structures, there are
finitely many formulas in $\Delta(\bar{x},\bar{y})$ up to logical
equivalence in $\mathcal{K}$. Thus, there is a finite conjunction
of atomic formulas $\alpha(\bar{x},\bar{y})$ such that 
\[
\mathcal{K}\vDash\alpha(\bar{x},\bar{y})\leftrightarrow\bigwedge\Delta(\bar{x},\bar{y}).
\]
Take $\mathbf{C}\in\mathcal{K}$ and suppose $\mathbf{C}\vDash\alpha(\bar{c},\bar{d})\wedge\alpha(\bar{c},\bar{e})$.
Then the maps $h,h':\mathbf{B}\rightarrow\mathbf{C}$, given by $h:\bar{a},\bar{b}\mapsto\bar{c},\bar{d}$
and $h':\bar{a},\bar{b}\mapsto\bar{c},\bar{e}$, are homomorphisms.
Since $h$ and $h'$ agree on $A$, it follows that $h=h'$. Hence
$\bar{d}=\bar{e}$, and we have shown that $\alpha(\bar{x},\bar{y})$
defines a function in $\mathcal{K}$.

(2)$\Rightarrow$(3). The p.p. formulas in (3) can be obtained by
adding existential quantifiers to the formulas given by (2).

(3)$\Rightarrow$(1). This is the same as (2)$\Rightarrow$(1) in
Theorem \ref{epic sii pp definable}.
\end{proof}
Again, it is worth noting that implications (2)$\Rightarrow$(3)$\Rightarrow$(1)
hold for any $\mathbf{A}$, $\mathbf{B}$ and $\mathcal{K}$.

The example below shows that, in the general case, the existential
quantifiers in (2) of Theorem \ref{epic sii pp definable} are necessary.
\begin{example}
Let $\mathbf{B}$ be the Browerian algebra whose lattice reduct is
depicted in Figure \ref{fig:alg tom}, and let $\mathbf{A}$ be the
subalgebra of $\mathbf{B}$ with universe $\{a_{0},a_{1},\dots\}\cup\{\top\}$.
It is proved in \cite[Thm. 6.1]{BezhanishviliMoraschiniRaftery} that
$\mathbf{A}\leq_{e}\mathbf{B}$ in $\mathbb{V}(\mathbf{B})$. We show
that (2) in Theorem \ref{thm:epicas en clases finitas} does not hold
for $\mathbf{A}$, $\mathbf{B}$ and $\mathbb{V}(\mathbf{B})$. Towards
a contradiction fix $d_{1}\in B\setminus A$, and suppose there are
a conjunction of equations $\alpha(x_{1},\dots,x_{n},y_{1},\dots,y_{m})$,
$c_{1},\dots,c_{n}\in A$ and $d_{2},\dots,d_{m}\in B$ such that
\begin{itemize}
\item $\alpha(\bar{x},\bar{y})$ defines a function in $\mathbb{V}(\mathbf{B})$
\item $\mathbf{B}\vDash\alpha(\bar{c},\bar{d})$.
\end{itemize}
Let $\mathbf{C}$ and $\mathbf{D}$ be the subalgebras of $\mathbf{B}$
generated by $\bar{c}$ and $\bar{c},\bar{d}$ respectively. Note
that $\mathbf{D}$ is finite and $\mathbf{C}<\mathbf{D}$. Also note
that $\alpha(\bar{x},\bar{y})$ defines a function in $\mathbb{V}(\mathbf{D})$,
and $\mathbf{D}\vDash\alpha(\bar{c},\bar{d})$, because $\alpha$
is quantifier-free. So we have $\mathbf{C}<_{e}\mathbf{D}$ in $\mathbb{V}(\mathbf{D})$;
but this is not possible, as Corollary 5.5 in \cite{BezhanishviliMoraschiniRaftery}
implies that there are no proper epic subalgebras in finitely generated
varieties of Browerian algebras.

\begin{figure}
\includegraphics{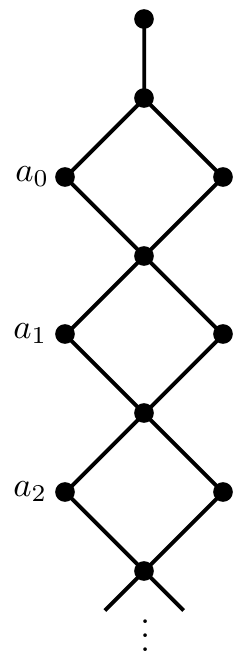}

\protect\caption{}

\label{fig:alg tom}
\end{figure}

\end{example}

\section{Checking for epic subalgebras in a subclass\label{sec:Checking-for-epic}}

In the current section all languages considered are algebraic, i.e.,
without relation symbols. Given a quasivariety $\mathcal{Q}$ it can
be a daunting task to determine whether \emph{$\mathcal{Q}$} has
surjective epimorphisms, or equivalently, no proper epic subalgebras.
In this section we prove two results that, under certain assumptions
on $\mathcal{Q}$, provide a (hopefully) more manageable class $\mathcal{S}\subseteq\mathcal{Q}$
such that $\mathcal{Q}$ has no proper epic subalgebras iff $\mathcal{S}$
has no proper epic subalgebras.

Our first result provides such a class $\mathcal{S}$ for quasivarieties
with a near-unanimity term. The second one for arithmetical varieties
whose class of finitely subdirectly irreducible members is universal.

\subsection{Quasivarieties with a near-unanimity term}

An $n$-ary term $t(x_{1},\dots,x_{n})$ is a \emph{near-unanimity}
term for the class $\mathcal{K}$ if $n\geq3$ and $\mathcal{K}$
satisfies the identities 
\[
t(x,\dots,x,y)=t(x,\dots,x,y,x)=\dots=t(y,x,\dots,x)=x.
\]
When $n=3$ the term $t$ is called a \emph{majority} term for $\mathcal{K}$.
In every structure with a lattice reduct the term $(x\vee y)\wedge(x\vee z)\wedge(y\vee z)$
is a majority term. This example is specially relevant since many
classes of structures arising from logic algebrizations have lattice
reducts.

For the sake of the exposition the results are presented for quasivarieties
with a majority term. They are easily generalized to quasivarieties
with an arbitrary near-unanimity term.

For functions $f:A\rightarrow A'$ and $g:B\rightarrow B'$ let $f\times g:A\times B\rightarrow A'\times B'$
be defined by $f\times g(a,b):=(f(a),g(b))$.
\begin{thm}[\cite{BP_para_clases}]
\label{thm:BP}Let $\mathcal{K}$ be a class of structures with a
majority term and suppose $\varphi(\bar{x},y)$ defines a function
in $\mathcal{K}$. T.f.a.e.:
\begin{enumerate}
\item There is a term $t(\bar{x})$ such that $\mathcal{K}\vDash\forall\bar{x},y\ \varphi(\bar{x},y)\rightarrow y=t(\bar{x}).$
\item For all $\mathbf{A},\mathbf{B}\in\mathbb{P}_{u}(\mathcal{K})$, all
$\mathbf{S}\leq\mathbf{A}\times\mathbf{B}$ and all $s_{1},\dots,s_{n}\in S$
such that $[\varphi]^{\mathbf{A}}\times[\varphi]^{\mathbf{B}}(\bar{s})$
is defined, we have that $[\varphi]^{\mathbf{A}}\times[\varphi]^{\mathbf{B}}(\bar{s})\in S$.
\end{enumerate}
\end{thm}
An algebra $\mathbf{A}$ in the quasivariety $\mathcal{Q}$ is \emph{relatively
subdirectly irreducible} provided its diagonal congruence is completely
meet irreducible in the lattice of $\mathcal{Q}$-congruences of $\mathbf{A}$.
We write $\mathcal{Q}_{RSI}$ to denote the class of relatively subdirectly
irreducible members of $\mathcal{Q}$. For a class $\mathcal{K}$
let $\mathcal{K}\times\mathcal{K}:=\{\mathbf{A}\times\mathbf{B}\mid\mathbf{A},\mathbf{B}\in\mathcal{K}\}$.
\begin{thm}
\label{thm:Testigos para Q con M}Let $\mathcal{Q}$ be a quasivariety
with a majority term and let $\mathcal{S}=\mathbb{P}_{u}(\mathcal{Q}_{RSI})$.
T.f.a.e.:
\begin{enumerate}
\item $\mathcal{Q}$ has surjective epimorphisms.
\item For all $\mathbf{A},\mathbf{B}\in\mathcal{Q}$ we have that $\mathbf{A}\leq_{e}\mathbf{B}$
in $\mathcal{Q}$ implies $\mathbf{A}=\mathbf{B}$.
\item For all $\mathbf{A},\mathbf{B}\in\mathbb{S}(\mathcal{S}\times\mathcal{S})$
we have that $\mathbf{A}\leq_{e}\mathbf{B}$ in $\mathcal{S}\times\mathcal{S}$
implies $\mathbf{A}=\mathbf{B}$.
\item $\mathbb{S}(\mathcal{S}\times\mathcal{S})$ has surjective epimorphisms.
\end{enumerate}
\end{thm}
\begin{proof}
The equivalences (1)$\Leftrightarrow$(2) and (3)$\Leftrightarrow$(4)
are immediate, and (2) clearly implies (3). We prove (3)$\Rightarrow$(2).
Suppose $\mathbf{A}\leq_{e}\mathbf{B}$ in $\mathcal{Q}$ and let
$b\in B$. We shall see that $b\in A$. By Theorem \ref{epic sii pp definable}
there is a p.p.\ $\mathcal{L}$-formula $\varphi(\bar{x},y)$ defining
a function in $\mathcal{Q}$, and such that $[\varphi]^{\mathbf{B}}(\bar{a})=b$
for some $\bar{a}\in A^{n}$. Let 
\[
\Sigma:=\{\varepsilon\mid\varepsilon\mbox{ is a p.p.\ formula of }\mathcal{L}_{A}\mbox{ and }\mathbf{B}_{A}\vDash\varepsilon\},
\]
and define 
\[
\mathcal{K}:=\{\mathbf{C}\in Mod(\Sigma)\mid\mathbf{C}|_{\mathcal{L}}\in\mathcal{S}\}.
\]
Let $\psi(y):=\varphi(\bar{a},y)$, and note that $\psi(y)$ defines
a nullary function in $\mathcal{K}$. Note as well that $\exists y\,\psi(y)\in\Sigma$,
and hence $[\psi]^{K}$ is defined for every $K\in\mathcal{K}$ .
We aim to apply Theorem \ref{thm:BP} to $\mathcal{K}$ and $\psi(y)$.
To this end fix $\mathbf{C},\mathbf{D}\in\mathbb{P}_{u}(\mathcal{K})=\mathcal{K}$
and let $\mathbf{S}\leq\mathbf{C}\times\mathbf{D}$. Note that as
$\Sigma$ is a set of p.p. formulas we have $\mathbf{C}\times\mathbf{D}\vDash\Sigma$,
and thus by Lemma \ref{toda pp implica homo} there is an ultrapower
$\mathbf{E}$ of $\mathbf{C}\times\mathbf{D}$ and a homomorphism
$h:\mathbf{B}_{A}\rightarrow\mathbf{E}$. We have that $\mathbf{E}\in\mathbb{P}_{u}(\mathcal{K}\times\mathcal{K})\subseteq\mathbb{P}_{u}(\mathcal{K})\times\mathbb{P}_{u}(\mathcal{K})=\mathcal{K}\times\mathcal{K}$,
and so 
\[
\mathbf{E}|_{\mathcal{L}}\in\mathcal{K}|_{\mathcal{L}}\times\mathcal{K}|_{\mathcal{L}}\subseteq\mathcal{S}\times\mathcal{S}.
\]
Next observe that since $h(\mathbf{A})\leq_{e}h(\mathbf{B})$ in $\mathcal{Q}$,
and $h(\mathbf{A}),h(\mathbf{B})\leq\mathbf{E}|_{\mathcal{L}}$, by
(3) it follows that $h(A)=h(B)$. Also, as $\mathbf{S}$ is an $\mathcal{L}_{A}$-subalgebra
of $\mathbf{E}$, we have that 
\[
h(\mathbf{B}_{A})=h(\mathbf{A}_{A})\leq\mathbf{S}.
\]
The fact that $\mathbf{B}\vDash\psi(b)$ implies $\mathbf{E}\vDash\psi(hb)$,
and so $[\psi]^{\mathbf{E}}=hb\in S$. We know that $\{\mathbf{C},\mathbf{D},\mathbf{C}\times\mathbf{D}\}\vDash\exists y\,\psi(y)$;
furthermore, since $\psi$ is p.p., we have $[\psi]^{\mathbf{C}}\times[\psi]^{\mathbf{D}}=[\psi]^{\mathbf{C}\times\mathbf{D}}$.
Putting all this together 
\[
[\psi]^{\mathbf{C}}\times[\psi]^{\mathbf{D}}=[\psi]^{\mathbf{C}\times\mathbf{D}}=[\psi]^{\mathbf{E}}\in S.
\]
Thus, Theorem \ref{thm:BP} produces an $\mathcal{L}_{A}$-term $t$
such that 
\begin{equation}
\mathcal{K}\vDash\forall y\ \psi(y)\rightarrow y=t.\label{eq:k sat quasi}
\end{equation}
In particular, for all $\mathbf{C}\in\mathcal{Q}_{RSI}$ and all $c_{1},\dots,c_{n}\in\mathbf{C}$
such that $[\varphi]^{\mathbf{C}}(\bar{c})$ is defined, we have 
\[
[\varphi]^{\mathbf{C}}(\bar{c})=t^{\mathbf{C}}(\bar{c}).
\]
Next let $\{\mathbf{B}_{i}\mid i\in I\}\subseteq\mathcal{Q}_{RSI}$
such that $\mathbf{B}\leq\prod_{I}\mathbf{B}_{i}$ is a subdirect
product. For every $i\in I$ let $\mathbf{B}_{i}^{A}$ be the expansion
of $\mathbf{B}_{i}$ to $\mathcal{L}_{A}$ given by $a^{\mathbf{B}_{i}^{A}}=\pi_{i}(a)$,
where $\pi_{i}:\mathbf{B}\rightarrow\mathbf{B}_{i}$ is the projection
map. It is clear that 
\begin{equation}
\mathbf{B}_{A}\leq\prod_{I}\mathbf{B}_{i}^{A}.\label{eq:B_A es subprod}
\end{equation}
Now, each $\mathbf{B}_{i}^{A}$ is a homomorphic image of $\mathbf{B}_{A}$,
so $\mathbf{B}_{i}^{A}\vDash\Sigma$ and thus $\mathbf{B}_{i}^{A}\in\mathcal{K}$
for all $i\in I$. Since $\forall y\ \psi(y)\rightarrow y=t$ is (equivalent
to) a quasi-identity, from (\ref{eq:k sat quasi}) and (\ref{eq:B_A es subprod})
we have 
\[
\mathbf{B}_{A}\vDash\forall y\ \psi(y)\rightarrow y=t.
\]
Hence $b=t^{\mathbf{B}_{A}}\in A$, and the proof is finished.
\end{proof}
Observe that Theorem \ref{thm:Testigos para Q con M} holds for any
$\mathcal{S}\subseteq\mathcal{Q}$ closed under ultraproducts and
containing $\mathcal{Q}_{RSI}$.
\begin{cor}
\label{cor:M+FG implica S(Q_RSI x Q_RSI) alcanza.}Let $\mathcal{Q}$
be a finitely generated quasivariety with a majority term. T.f.a.e.:
\begin{enumerate}
\item $\mathcal{Q}$ has surjective epimorphisms.
\item $\mathcal{\mathbb{S}}(\mathcal{Q}_{RSI}\times\mathcal{Q}_{RSI})$
has surjective epimorphisms.
\end{enumerate}
\end{cor}
\begin{proof}
For any class $\mathcal{K}$ we have $\mathbb{Q}(\mathcal{K})_{RSI}\subseteq\mathbb{ISP}_{u}(\mathcal{K})$.
Thus if $\mathcal{Q}$ is finitely generated, then $\mathcal{Q}_{RSI}$
is (up to isomorphic copies) a finite set of finite algebras, and
the corollary follows at once from Theorem \ref{thm:Testigos para Q con M}.
\end{proof}
Recall that an algebra $\mathbf{A}$ is \emph{finitely subdirectly
irreducible} if its diagonal congruence is meet irreducible in the
congruence lattice of $\mathbf{A}$. It is \emph{subdirectly irreducible
}if the diagonal is completely meet irreducible. For a variety $\mathcal{V}$
we write ($\mathcal{V}_{FSI}$) $\mathcal{V}_{SI}$ to denote its
class of (finitely) subdirectly irreducible members.

An interesting consequence of Corollary \ref{cor:M+FG implica S(Q_RSI x Q_RSI) alcanza.}
is the following.
\begin{cor}
\label{cor:NU implica decidible}Let $\mathcal{F}$ be a finite set
of finite algebras with a common majority term. It is decidable whether
the \textup{(}quasi\textup{)}variety generated by $\mathcal{F}$ has
surjective epimorphisms.\end{cor}
\begin{proof}
Let $\mathcal{V}$ be the variety generated by $\mathcal{F}$. By
J\'onsson's lemma \cite{jonssons_lemma} $\mathcal{V}_{SI}\subseteq\mathbb{HSP}_{u}(\mathcal{F})=\mathbb{HS}(\mathcal{F})$
is a finite set of finite structures, and by Corollary \ref{cor:M+FG implica S(Q_RSI x Q_RSI) alcanza.}
it suffices to decide whether $\mathcal{\mathbb{S}}(\mathcal{V}_{SI}\times\mathcal{V}_{SI})$
has surjective epimorphisms, and this is clearly a decidable problem.
If $\mathcal{Q}$ is the quasivariety generated by $\mathcal{F}$,
then $\mathcal{Q}_{RSI}\subseteq\mathbb{ISP}_{u}(\mathcal{F})=\mathbb{IS}(\mathcal{F})$,
and the same reasoning applies.
\end{proof}

\subsection{Arithmetical varieties whose FSI members form a universal class}

A variety $\mathcal{V}$ is \emph{arithmetical} if for every $\mathbf{A}\in\mathcal{V}$
the congruence lattice of $\mathbf{A}$ is distributive and the join
of any two congruences is their composition. For example, the variety
of boolean algebras is arithmetical.
\begin{lem}
\label{lem:termino interpolante para pp en aritmeticas} Let $\mathcal{V}$
be an arithmetical variety such that $\mathcal{V}_{FSI}$ is a universal
class, and let $\varphi(\bar{x},y)$ be a p.p. formula defining a
function in $\mathcal{V}$. Suppose that for all $\mathbf{A}\in\mathcal{V}_{FSI}$,
all $\mathbf{S}\leq\mathbf{A}$ and all $s_{1},\dots,s_{n}\in S$
such that $\mathbf{A}\vDash\exists y\,\varphi(\bar{s},y)$, we have
$\mathbf{S}\vDash\exists y\,\varphi(\bar{s},y)$. Then there is a
term $t(\bar{x})$ such that $\mathcal{V}\vDash\forall\bar{x},y\ \varphi(\bar{x},y)\rightarrow y=t(\bar{x}).$\end{lem}
\begin{proof}
Add new constants $c_{1},\dots,c_{n}$ to the language of $\mathcal{V}$
and let $\mathcal{K}:=\{(\mathbf{A},\bar{a})\mid\mathbf{A}\vDash\exists y\,\varphi(\bar{c},y)\mbox{ and }\mathbf{A}\in\mathcal{V}_{FSI}\}$.
Note that $\psi(y):=\varphi(\bar{c},y)$ defines a nullary function
in $\mathcal{K}$, and this function is defined for every member of
$\mathcal{K}$. Also note that by our assumptions $\mathcal{K}$ is
a universal class. Using J\'onsson's lemma \cite{jonssons_lemma} it
is not hard to show that $\mathbb{V}(\mathcal{K})_{FSI}=\mathcal{K}$.
Since $\mathcal{K}|_{\mathcal{L}}$ is contained in an arithmetical
variety it has a Pixley Term \cite[Thm. 12.5]{BurrisSankappanavar1981},
which also serves as a Pixley Term for $\mathcal{K}$, and thus $\mathbb{V}(\mathcal{K})$
is arithmetical. Next we show that $\psi(y)$ is equivalent to a positive
open formula in $\mathcal{K}$. By \cite[Thm. 3.1]{lemas_semanticos}
it suffices to show that
\begin{itemize}
\item For all $\mathbf{A},\mathbf{B}\in\mathcal{K}$, all $\mathbf{S}\leq\mathbf{A}$,
all $h:\mathbf{S}\rightarrow\mathbf{B}$ and every $a\in A$ we have
that $\mathbf{A}\vDash\psi(a)$ implies $\mathbf{B}\vDash\psi(ha)$.
\end{itemize}
So suppose $\mathbf{A}\vDash\psi(a)$. From our hypothesis and the
fact that $\psi(y)$ defines a function we have $\mathbf{S}\vDash\psi(a)$,
and as $\psi(y)$ is p.p. we obtain $\mathbf{B}\vDash\psi(ha)$. Hence
there is a positive open formula $\beta(y)$ equivalent to $\psi(y)$
in $\mathcal{K}$. Now, \cite[Thm. 2.3]{CzelakowskiDziobiak1990}
implies that there is a conjunction of equations $\alpha(y)$ equivalent
to $\beta(y)$ (and thus to $\psi(y)$) in $\mathcal{K}$. We have
$\mathcal{K}\vDash\exists!y\,\alpha(y)$, and by \cite[Lemma 7.8]{lemas_semanticos}
there is an $\mathcal{L}\cup\{c_{1},\dots,c_{n}\}$-term $t'$ such
that $\mathbb{V}(\mathcal{K})\vDash\alpha(t')$. Let $t(x_{1},\dots,x_{n})$
be an $\mathcal{L}$-term such that $t'=t(\bar{c})$. So, if $\Gamma$
is a set of axioms for $\mathcal{V}_{FSI}$, we have
\[
\Gamma\cup\{\exists y\,\varphi(\bar{c},y)\}\vDash\varphi(\bar{c},t(\bar{c})),
\]
and this implies
\[
\Gamma\vDash\exists y\,\varphi(\bar{c},y)\rightarrow\varphi(\bar{c},t(\bar{c})),
\]
or equivalently
\[
\mathcal{V}_{FSI}\vDash\forall y(\varphi(\bar{c},y)\rightarrow\varphi(\bar{c},t(\bar{c}))).
\]
This and the fact that that $\varphi(\bar{x},y)$ defines a function
in $\mathcal{V}$ yields
\[
\mathcal{V}_{FSI}\vDash\forall\bar{x},y\ \varphi(\bar{x},y)\rightarrow y=t(\bar{x}).
\]
To conclude, note that $\forall\bar{x},y\ \varphi(\bar{x},y)\rightarrow y=t(\bar{x})$
is logically equivalent to a quasi-identity, and since it holds in
$\mathcal{V}_{FSI}$ it must hold in $\mathcal{V}$.\end{proof}
\begin{thm}
\label{thm:Testigos para V aritmetica}Let $\mathcal{V}$ be an arithmetical
variety such that $\mathcal{V}_{FSI}$ is a universal class T.f.a.e.:
\begin{enumerate}
\item $\mathcal{V}$ has surjective epimorphisms.
\item For all $\mathbf{A},\mathbf{B}\in\mathcal{V}$ we have that $\mathbf{A}\leq_{e}\mathbf{B}$
in $\mathcal{V}$ implies $\mathbf{A}=\mathbf{B}$.
\item For all $\mathbf{A},\mathbf{B}\in\mathcal{V}_{FSI}$ we have that
$\mathbf{A}\leq_{e}\mathbf{B}$ in $\mathcal{V}_{FSI}$ implies $\mathbf{A}=\mathbf{B}$.
\item $\mathcal{V}_{FSI}$ has surjective epimorphisms.
\end{enumerate}
\end{thm}
\begin{proof}
We prove (3)$\Rightarrow$(2) which is the only nontrivial implication.
Suppose $\mathbf{A}\leq_{e}\mathbf{B}$ in $\mathcal{V}$ and let
$b\in B$. We shall see that $b\in A$. By Theorem \ref{epic sii pp definable}
there is a p.p.\ $\mathcal{L}$-formula $\varphi\left(\bar{x},y\right)$
defining a function in $\mathcal{V}$, and such that $[\varphi]^{\mathbf{B}}(\bar{a})=b$
for some $\bar{a}\in A^{n}$. Let 
\[
\Sigma:=\{\varepsilon\mid\varepsilon\mbox{ is a p.p.\ sentence of }\mathcal{L}_{A}\mbox{ and }B_{A}\vDash\varepsilon\},
\]
and define 
\[
\mathcal{K}:=\{\mathbf{C}\in Mod(\Sigma)\mid\mathbf{C}|_{\mathcal{L}}\in\mathcal{V}_{FSI}\}.
\]

\begin{claim*}
\label{cl: K universal}$\mathcal{K}$ is a universal class.
\end{claim*}
Since $\mathcal{K}$ is axiomatizable we only need to check that $\mathcal{K}$
is closed under substructures. Let $\mathbf{C}\leq\mathbf{D}\in\mathcal{K}$;
clearly $\mathbf{C}|_{\mathcal{L}}\in\mathcal{V}_{FSI}$, so it remains
to see that $\mathbf{C}\vDash\Sigma$. As $\mathbf{D}\vDash\Sigma$,
Lemma \ref{toda pp implica homo} yields a homomorphism $h:\mathbf{B}_{A}\rightarrow\mathbf{E}$
with $\mathbf{E}$ an ultrapower of $\mathbf{D}$. Note that $\mathbf{E}\in\mathcal{K}$.
Since $h(\mathbf{A})\leq_{e}h(B)$ in $\mathcal{V}$ and $h(\mathbf{A}),h(\mathbf{B})\in\mathcal{V}_{FSI}$,
it follows that $h(A)=h(B)$, because there are no proper epic subalgebras
in $\mathcal{V}_{FSI}$. Now $\mathbf{C}$ is an $\mathcal{L}_{A}$-subalgebra
of $\mathbf{D}$, so $h(B)=h(A)\subseteq C$. Finally, since $h(\mathbf{B})\vDash\Sigma$
and every sentence in $\Sigma$ is existential, we obtain $\mathbf{C}\vDash\Sigma$.
This finishes the proof of the claim.
\begin{claim*}
$\mathbb{V}(\mathcal{K})$ is arithmetical and $\mathbb{V}(\mathcal{K})_{FSI}=\mathcal{K}$.
\end{claim*}
To show that $\mathbb{V}(\mathcal{K})$ is arithmetical we can proceed
as in the proof of Lemma \ref{lem:termino interpolante para pp en aritmeticas}.
We prove $\mathbb{V}(\mathcal{K})_{FSI}=\mathcal{K}$. Note that for
$\mathbf{C}\in\mathcal{K}$ we have that $\mathbf{C}$ and $\mathbf{C}|_{\mathcal{L}}$
have the same congruences; hence every algebra in $\mathcal{K}$ is
FSI. For the other inclusion, J\'onsson's lemma \cite{jonssons_lemma}
produces $\mathbb{V}(\mathcal{K})_{FSI}\subseteq\mathbb{HSP}_{u}(\mathcal{K})$,
and by the first claim $\mathbb{HSP}_{u}(\mathcal{K})=\mathbb{H}(\mathcal{K})$.
So, as $\mathbb{H}(\mathcal{K})\vDash\Sigma$, we have that $\mathbb{V}(\mathcal{K})_{FSI}\vDash\Sigma$
and thus $\mathbb{V}(\mathcal{K})_{FSI}\subseteq\mathcal{K}$.

Next we want to apply Lemma \ref{lem:termino interpolante para pp en aritmeticas}
to $\mathbb{V}(\mathcal{K})$ and $\varphi(\bar{a},y)$, so we need
to check that the hypothesis hold. Take $\mathbf{C}\in\mathcal{K}$
and $\mathbf{S}\leq\mathbf{C}$. Since $\mathcal{K}$ is universal
we have $\mathbf{S}\in\mathcal{K}$, and thus $\mathbf{S}\vDash\exists y\,\varphi(\bar{a},y)$.
Let $t$ be a term such that $\mathbb{V}(\mathcal{K})\vDash\forall y\ \varphi(\bar{a},y)\rightarrow y=t.$
Then $b=t^{\mathbf{B}_{A}}\in A$, and we are done.
\end{proof}
Every discriminator variety (see \cite[Def. 9.3]{BurrisSankappanavar1981}
for the definition) satisfies the hypothesis in Theorem \ref{thm:Testigos para V aritmetica}.
Furthermore, in such a variety every FSI member is simple (i.e., has
exactly two congruences). Writing $\mathcal{V}_{S}$ for the class
of simple members in $\mathcal{V}$ we have the following immediate
consequence of Theorem \ref{thm:Testigos para V aritmetica}.
\begin{cor}
For a discriminator variety $\mathcal{V}$ the following are equivalent.
\begin{enumerate}
\item $\mathcal{V}$ has surjective epimorphisms.
\item For all $\mathbf{A},\mathbf{B}\in\mathcal{V}$ we have that $\mathbf{A}\leq_{e}\mathbf{B}$
in $\mathcal{V}$ implies $\mathbf{A}=\mathbf{B}$.
\item For all $\mathbf{A},\mathbf{B}\in\mathcal{V}_{S}$ we have that $\mathbf{A}\leq_{e}\mathbf{B}$
in $\mathcal{V}_{S}$ implies $\mathbf{A}=\mathbf{B}$.
\item $\mathcal{V}_{S}$ has surjective epimorphisms.
\end{enumerate}
\end{cor}
It is not uncommon for a variety arising as the algebrization of a
logic to be a discriminator variety; thus the above corollary could
prove helpful in establishing the Beth definability property for such
a logic.

Another special case relevant to algebraic logic to which Theorem
\ref{thm:Testigos para V aritmetica} applies is given by the class
of Heyting algebras and its subvarieties (none of these are discriminator
varieties with the exception of the class of boolean algebras). Heyting
algebras constitute the algebraic counterpart to intuitionistic logic,
and have proven to be a fertile ground to investigate definability
and interpolation properties of intuitionistic logic and its axiomatic
extensions by algebraic means (see \cite{BezhanishviliMoraschiniRaftery}
and its references).

\thanks{I would like to thank Diego Casta\~no and Tommaso Moraschini for their
insightful discussions during the preparation of this paper.}

\bibliographystyle{plain}
\bibliography{epic}

\end{document}